\documentclass[11pt]{amsart}

\usepackage[margin=1in]{geometry}

\usepackage{graphicx}
\usepackage{enumerate}
\usepackage{mathtools}
\usepackage{amsthm}
\usepackage{algorithm}
\usepackage{algorithmic}
\usepackage{booktabs}
\usepackage{subcaption}
\usepackage{stmaryrd}

\graphicspath{{./images/}}

\usepackage{hyperref}
\usepackage{xcolor}
\hypersetup{
    colorlinks,
    linkcolor={red!50!black},
    citecolor={blue!50!black},
    urlcolor={blue!80!black}
}

\DeclareMathOperator{\dom}{dom}
\DeclareMathOperator{\prox}{prox}
\DeclareMathOperator*{\argmin}{argmin}

\newcommand{\grad}{\nabla}

\let\S\undefine
\let\SS\undefine
\newcommand{\tp}{{\scriptscriptstyle\mathsf{T}}}
\newcommand{\qS}{{\scriptscriptstyle\mathsf{qS}}}
\newcommand{\qSBB}{{\scriptscriptstyle\mathsf{qSBB}}}
\newcommand{\S}{{\scriptscriptstyle\mathsf{S}}}
\newcommand{\SBB}{{\scriptscriptstyle\mathsf{SBB}}}
\newcommand{\SSBB}{{\scriptscriptstyle\mathsf{SSBB}}}
\newcommand{\SS}{{\scriptscriptstyle\mathsf{SS}}}

\theoremstyle{definition}
\newtheorem{theorem}{Theorem}[section]

\newtheorem{lemma}[theorem]{Lemma}
\newtheorem{corollary}[theorem]{Corollary}
\newtheorem{proposition}[theorem]{Proposition}
\newtheorem{assumption}[theorem]{Assumption}

\numberwithin{equation}{section}

\begin{document}

\title{Stochastic Steffensen Method}
\author{Minda Zhao}
\author{Zehua Lai}
\author{Lek-Heng Lim}
\address{Computational and Applied Mathematics, University of Chicago, Chicago, IL 60637}
\email{mindazhao@uchicago.edu, laizehua@uchicago.edu, lekheng@uchicago.edu}
\begin{abstract}
Is it possible for a first-order method, i.e., only first derivatives allowed, to be quadratically convergent? For univariate loss functions, the answer is yes --- the \emph{Steffensen method} avoids second derivatives and is still quadratically convergent like Newton method. By incorporating an optimal step size we can even push its convergence order beyond quadratic to $1+\sqrt{2} \approx 2.414$. While such high convergence orders are a pointless overkill for a deterministic algorithm, they become rewarding when the algorithm is randomized for problems of massive sizes, as randomization invariably compromises convergence speed. We will introduce two adaptive learning rates inspired by the Steffensen method, intended for use in a stochastic optimization setting and requires no hyperparameter tuning aside from batch size. Extensive experiments show that they compare favorably with several existing first-order methods. When restricted to a quadratic objective, our stochastic Steffensen methods reduce to randomized Kaczmarz method --- note that this is not true for SGD or SLBFGS --- and thus we may also view our methods as a generalization of randomized Kaczmarz to arbitrary objectives.
\end{abstract}
\maketitle

\section{Introduction}\label{sec:intro}

In minimizing a \emph{univariate} function $f$ with an iteration $x_{k+1} = x_k - f'(x_k)/g(x_k)$, possibilities for $g$ include
\begin{alignat*}{7}
&\text{gradient:}\quad & g(x_k) &= 1,  & &\text{Newton:} &g(x_k) &= f''(x_k),\\
&\text{secant:}  &g(x_k) &= \frac{f'(x_k) - f'(x_{k-1})}{x_k - x_{k-1}},   \qquad &&\text{Steffensen:} \quad &g(x_k) &= \frac{f'(x_k + f'(x_k)) - f'(x_k)}{f'(x_k)} ,
\end{alignat*}
with different orders of convergence $q$, i.e., $\lvert x_{k+1} - x^* \rvert \le c\lvert x_k - x^* \rvert^q$. Gradient descent has $q =1$, secant method $q  = (1+\sqrt{5})/2$, Newton and Steffensen methods both have $q = 2$.

Steffensen method \cite{Steff, Steff2} is a surprise. Not only does it not require second derivatives (like Newton) to achieve quadratic convergence, it also does not achieve its superior convergence through the use of multisteps (like secant). In other words, the $k$th Steffensen iterate only depends on $x_k$ but not $x_{k-1}, x_{k-2}$, etc.

Nevertheless, while the other three methods have widely used multivariate generalizations (secant method has several, as quasi-Newton methods, as Barzilai--Borwein step size, etc), all existing multivariate generalizations of Steffensen method \cite{Amat2016steffensen,
Ezquerro2014hybrid,Henrici,
Huang1970unified,
Johnson1968steffensen,
Nedzhibov2011approach,
Nievergelt1991Aitken,
Nievergelt1995condition,
Noda1981Aitken1,
Noda1986Aitken2,
Noda1986Aitken3,
Noda1990Aitken4,
Noda1992Aitken5} involve multivariate \emph{divided differences} that require $O(n^2)$ function evaluations and are no less expensive than using the full Hessian. Furthermore these multivariate generalizations are no longer one-step methods. As a result they have not found widespread use.

Our contributions are as follows:
\begin{enumerate}[\upshape (i)]
\item We show that by incorporating an optimal step size parameter the convergence of Steffensen method may be further improved beyond quadratic to $q = 1 + \sqrt{2}$.

\item We extend Steffensen method to a multivariate setting as an adaptive learning rate, avoiding divided differences, requiring just two gradient evaluations, and remaining a one-step method.

\item\label{it:rand} We show that when used in a randomized setting, our methods outperform SGD, SVRG, and SLBFGS on a variety of standard machine learning tasks on real data sets.
\end{enumerate}
The performance in \eqref{it:rand} is measured in actual running time. But aside from speed, our methods have two advantages over SLBFGS, which has become a gold standard in machine learning.
\begin{enumerate}[\upshape (a)]
\item Quasi-Newton methods may involve matrix-vector product, a two-loop recursion with $O(d^2)$ computation. Although deterministic LBFGS does not form matrix-vector product explicitly, stochastic LBFGS does. Our multivariate Steffensen method, whether deterministic or stochastic, is free of such products. 

\item Quasi-Newton methods come in two flavors: Hessian or inverse Hessian updates. The latter seems  a nobrainer as it avoids matrix inversion but this is a fallacy. It is common knowledge among practitioners \cite[Section~4.5.2.2]{gill2019practical} that the inverse Hessian version often conceals an ill-conditioned approximate Hessian; one should instead update the Cholesky factors of the approximate Hessian in order to detect ill-conditioning. By its design, LBFGS  inevitably uses the inverse Hessian version. Our multivariate Steffensen methods are not quasi-Newton methods and do not involve approximate Hessians, avoiding this issue entirely.
\end{enumerate}

Johan Steffensen first proposed his eponymous method  \cite{Steff} in 1933. See \cite{Brez} for an informative history of the method and a biography of its inventor. The method was described in the classic books of Henrici \cite[pp.~91--95]{Henrici} and Householder \cite[p.~164]{House}  but has remained more of a textbook curiosity. One reason, as we mentioned above and will elaborate in Section~\ref{sec:multvar}, is that there has been no viable multivariate version.

Another reason, as we will speculate, is that much like the Kaczmarz method \cite{Kaczmarz,KaczEng} for iterative solution of linear systems had lingered in relative obscurity until it was randomized \cite{strohmer2009randomized}, Steffensen method is also most effective in a randomized setting. This is in fact more than an analogy; we will show in Section~\ref{sec:kacz} that the stochastic Steffensen method we propose reduces to randomized Kaczmarz method when applied to a quadratic objective --- not true for SGD, SVRG, or SLBFGS. So one may also view our stochastic Steffensen method as a generalization of randomized Kaczmarz method to arbitrary differentiable objective functions. In Section~\ref{sec:nondiff}, we show that differentiability may be dropped and in Section~\ref{sec:conv} we supply proofs of linear convergence.

Stochastic optimization has grown into a vast subject. We have limited our comparison in this article to stochastic variants of classical methods that rely primarily on \emph{gradients}. In the numerical experiments in Section~\ref{sec:num}, we will see that the stochastic Steffensen methods compare favorably with SGD, SVRG (with or without Barzilai--Borwein step size), and SLBFGS across different tasks in the \texttt{LIBSVM} datasets: ridge regression, logistic regression, and support vector machines with squared hinge loss. We did not include more sophisticated stochastic optimization algorithms that bring in additional features like \emph{moments} \cite{duchi2011adaptive,hinton2012neural,kingma2014adam} or \emph{momentum} \cite{Nesterov,Poljak,qian1999momentum,reddi2019convergence} for two reasons. Firstly these more sophisticated algorithms invariably require heavy tuning compared to purely gradient-based methods. Secondly we view them as enhancements to gradients-based methods and our proposed stochastic Steffensen methods likewise lend themselves to such enhancements. As such, the most appropriate and equitable comparisons for us would be the aforementioned gradient-based methods.

\subsection*{Background}

As in the usual setting for stochastic gradient descent and its variants, our goal is to minimize an objective function of the form
\begin{equation}\label{large sum}
f(x)= \frac{1}{n} \sum_{i=1}^n f_i(x),
\end{equation}
where $x\in \mathbb{R}^d$ is the model parameter. Such functions are ubiquitous in machine learning, arising from the emperical risk minimization (ERM) problem where $f_i$ takes the form
\[
f_i(x) = \ell(w_i^\tp x; y_i) + \lambda R(x),
\]
with $\ell:\mathbb{R}\times\mathbb{R}\to\mathbb{R}_+$ the loss function, $R: \mathbb{R}^d\to \mathbb{R}_+$ the regularizer, $\lambda \ge 0$ the regularization parameter, and $\{(w_i,y_i)\in \mathbb{R}^d \times \mathbb{R} :i=1,\dots, n\}$ the training set with labels. Different choices of $\ell$ and $R$ give $l^2$-regularized logistic regression, lasso regression, soft-margin support vector machine, etc.

The challenge here is that the dimension $d$ and sample size $n$ are extremely large in modern situations, mandating the use of \emph{first-order methods} that rely only on first derivatives. But when $n$ is large, even computing the full gradient of all $f_1,\dots, f_n$ is intractable, and we need \emph{stochastic optimization methods} that update $x$ only after processing a small subset of data, permitting progress in the time deterministic methods make only a single step. Consequently, stochastic first-order methods have become the method of choice, with stochastic gradient descent (SGD) \cite{Robbins1951sto} and its many variants \cite{roux2012stochastic, defazio2014saga, johnson2013accelerating} and various stochastic quasi-Newton methods \cite{moritz2016linearly, Byrd2016A, zhao2018stochastic} ruling the day.

\subsection*{Conventions}

In this article, we use the terms \emph{learning rate} and \emph{step size} slightly differently. Take for example our Steffensen--Barzilai--Borwein iteration in \eqref{eq:SBB}:
\[
x_{k+1} = x_k - \frac{\beta_k \|\grad f(x_k)\|^2}{[\grad f(x_k + \beta_k \grad f(x_k)) - \grad f(x_k)]^\tp \grad f(x_k)} \grad f(x_k),
\]
the coefficient
\[
\eta^{\SBB}_k \coloneqq \frac{\beta_k \|\grad f(x_k)\|^2}{[\grad f(x_k + \beta_k \grad f(x_k)) - \grad f(x_k)]^\tp \grad f(x_k)}
\]
will be called a learning rate whereas the coefficient
\[
\beta_k \coloneqq \frac{\|x_k - x_{k-1}\|^2}{[\grad f(x_k) - \grad f(x_{k-1})]^\tp(x_k - x_{k-1})}
\]
will be called a step size. In general, the term `learning rate' will be used exclusively to refer to the coefficient of a search direction, which may be a gradient, a stochastic gradient, a variance-reduced stochastic gradient, etc. The term `step size' will be used for coefficients in other contexts like $\beta_k$ in the definition of the learning rate $\eta^{\SBB}_k$.

We will use $\eta_k$ to denote a general learning rate.  For the learning rate of a particular algorithm, we will indicate the algorithm in superscript. For example, $\eta^{\SBB}_k$ above is the learning rate of Steffensen--Barzilai--Borwein method (SBB). The Barzilai--Borwein step size above will always be denoted $\beta_k$ throughout.

\section{Stochastic Multivariate Steffensen Methods}\label{sec:derive}

Our three-step strategy is to (a) push the convergence order of the univariate Steffensen method to its limit, (b) extend the resulting method to a multivariate setting, and then (c) randomize the multivariate algorithm. For (a), we are led naturally to the Barzilai--Borwein step size; for (b), we emulate the multivariate extension of secant method into quasi-Newton method; and for (c), we draw inspiration from stochastic gradient descent and its various derivatives.

\subsection{Deterministic univariate setting}\label{sec:uni}

As we saw in Section~\ref{sec:intro}, univariate Steffensen method:
\begin{equation}\label{uniSteff}
x_{k+1} = x_k - \frac{f'(x_k)^2}{f'(x_k + f'(x_k)) - f'(x_k)}
\end{equation}
avoids second-order derivatives and yet preserves quadratic convergence with the use of two first-order derivatives $f'(x_k + f'(x_k))$ and $f'(x_k)$.  With modern hindsight, it is clear that we may obtain an immediate improvement in \eqref{uniSteff}, one that is essentially free, by incorporating a coefficient $\beta_k$ that only depends on quantities already computed. The analysis in the next two results will lead us to an appropriate choice of $\beta_k$. Note that although the algorithms require only first derivatives of $f$, the convergence results assume that $f$ has a higher degree of smoothness.

\begin{proposition}[Convergence order of Steffensen method]\label{univariate steffensen}
Let $f$ be a function that is $C^3$ in a neighborhood of a stationary point $x^*$ with $f'(x^*) = 0$ and $f''(x^*) \neq 0$. Let  $\alpha \in \mathbb{R}$ be a nonzero constant parameter and
\[
x_{k+1} \coloneqq x_k - \frac{\alpha f'(x_k)^2}{f'\bigl(x_k + \alpha f'(x_k)\bigr) - f'(x_k)}
\]
for $k = 0,1,2,\dots.$ If $\lim_{k \to \infty} x_k = x^*$, then
\[
\lim_{k\to\infty}\frac{|\varepsilon_{k+1}|}{|\varepsilon_k^2|}=\frac{1}{2}\left|\frac{f'''(x^*)}{f''(x^*)}\right|\left|1 + \alpha f''(x^*)\right|,
\]
where $\varepsilon_k \coloneqq x_k - x^*$ denotes the error in iteration $k$.
\end{proposition}

\begin{proof}
Let $\varepsilon_k = x_k - x^*$. Subtracting $x^*$ from both sides, we get
\[
\varepsilon_{k+1} = \varepsilon_{k} - \frac{\alpha f'(x_k)^2}{f'(x_k + \alpha f'(x_k)) - f'(x_k)}.
\]
Taylor expanding $f'(x_k + \alpha f'(x_k))$ about $x_k$, we get
\[
f'(x_k + \alpha f'(x_k)) = f'(x_k) + f''(x_k)\alpha f'(x_k) + \frac{f'''(\xi_k)}{2}\alpha^2f'(x_k)^2
\]
for some $\xi_k$  between $x_k$ and $x_k + \eta f'(x_k)$. Combining the previous two equations, we have
\begin{equation}\label{equ:iter1}
\varepsilon_{k+1} = \varepsilon_{k} - \frac{f'(x_k)}{f''(x_k) + \frac{f'''(\xi_k)}{2}\alpha f'(x_k)} = \frac{-f'(x_k) + f''(x_k)\varepsilon_k + \frac12 f'''(\xi_k)\alpha f'(x_k)\varepsilon_k}{f''(x_k) + \frac12 f'''(\xi_k)\alpha f'(x_k)}.
\end{equation}
Taylor expanding $f'$ about $x_k$, we get
\[
0 = f'(x^*) = f'(x_k) - f''(x_k)\varepsilon_k + \frac{f'''(\xi_k^*)}{2}\varepsilon_k^2
\]
for some $\xi_k^*$ between $x_k$ and $x^*$. Plugging $f'(x_k)$ into \eqref{equ:iter1} gives us
\[
\varepsilon_{k+1} = \frac{f'''(\xi_k^*)\varepsilon_k^2 + \alpha f'''(\xi_k)f''(x_k)\varepsilon_k^2 - \frac{\alpha}{2} f'''(\xi_k)f'''(\xi_k^*)\varepsilon_k^3}{2f''(x_k) + f'''(\xi_k)\alpha f'(x_k)}.
\]
Taking limit $k \to \infty$ and using continuity of $f'$, $f''$, and $f'''$ at $x^*$, we have
\begin{align*}
\lim_{k\to\infty}\frac{|\varepsilon_{k+1}|}{|\varepsilon_k^2|} &= \lim_{k\to\infty}\biggl|\frac{f'''(\xi_k^*) + \alpha f'''(\xi_k)f''(x_k) - \frac{\alpha}{2} f'''(\xi_k)f'''(\xi_k^*)\varepsilon_k}{2f''(x_k) + f'''(\xi_k)\alpha f'(x_k)}\biggr|\\
& =\frac{1}{2}\biggl|\frac{f'''(x^*)}{f''(x^*)}\biggr| |1 + \alpha f''(x^*)|
\end{align*}
as required.
\end{proof}

We next show that with an appropriate choice of $\alpha$, we can push Steffensen method into the superquadratically convergent regime. The quadratic convergence in Proposition~\ref{univariate steffensen} is independent of the value $\alpha$ and we may thus choose a different $\alpha$ at every step. Of course if we simply set $\alpha_k = -1/f''(x_k)$ in Proposition~\ref{univariate steffensen}, we will obtain a cubically convergent algorithm. However since we want a first-order method whose learning rate depends only on previously computed quantities, we set $\alpha_k = -(x_k - x_{k-1})/[f'(x_k) - f'(x_{k-1})]$ to be the finite difference to avoid second derivatives --- as it turns out, this improves convergence order to $1 + \sqrt{2}$.

\begin{theorem}[Convergence order of Steffensen method with Barzilai--Borwein step size]\label{superQ}
Let $f$ be a function that is $C^4$ in a neighborhood of a stationary point $x^*$ with $f'(x^*) = 0$ and $f''(x^*) \neq 0$.  Let
\[
\beta_k =-\frac{x_k - x_{k-1}}{f'(x_k) - f'(x_{k-1})}
\]
and
\begin{equation}\label{eq:bbs}
x_{k+1} = x_k - \frac{\beta_k f'(x_k)^2}{f'\bigl(x_k + \beta_k f'(x_k)\bigr) - f'(x_k)}
\end{equation}
for $k =0,1,2,\dots.$ If $\lim_{k \to \infty} x_k \to x^*$, then
\[
\lim_{k\to\infty} \frac{|\varepsilon_{k+1}|}{|\varepsilon_k^2\varepsilon_{k-1}|} = \Bigl(\frac{f'''(x^*)}{2f''(x^*)}\Bigr)^2.
\]
In particular, the order of convergence of \eqref{eq:bbs} is superquadratic with $1+\sqrt{2} \approx 2.414$.
\end{theorem}

\begin{proof}
Taylor expanding $f'(x_k+\beta_kf'(x_k))$ at $x_k$, we get
\[
f'(x_k + \beta_k f'(x_k)) =f'(x_k) + f''(x_k)\beta_k f'(x_k) +\frac{f^{(3)}(x_k)}{2}\beta_k^2f'(x_k)^2 + \frac{f^{(4)}(\xi_k)}{6}\beta_k^3f'(x_k)^3
\]
for some $\xi_k$ between $x_k$ and $x_k + \eta_kf'(x_k)$. Let $\varepsilon_k = x_k - x^*$, we have
\begin{equation}\label{equ:iter2}
\begin{aligned}
\varepsilon_{k+1} &= \varepsilon_{k} - \frac{f'(x_k)}{f''(x_k) + \frac12 f^{(3)}(x_k) \beta_k f'(x_k) + \frac16 f^{(4)}(\xi_k)\beta_k^2f'(x_k)^2}\\
&= \frac{-f'(x_k)+f''(x_k)\varepsilon_k + \frac12 f^{(3)}(x_k)\beta_kf'(x_k)\varepsilon_k+\frac16 f^{(4)}(\xi_k) \beta_k^2f'(x_k)^2\varepsilon_k}{f''(x_k) + \frac12 f^{(3)}(x_k) \beta_k f'(x_k) + \frac16 f^{(4)}(\xi_k)\beta_k^2f'(x_k)^2}.
\end{aligned}
\end{equation}
Taylor expanding $f'(x^*)$ at $x_k$ to $4$th, $3$th, and $2$nd order, we get
\begin{align*}
0 = f'(x^*) &= f'(x_k) - f''(x_k)\varepsilon_k + \frac{f^{(3)}(x_k)}{2}\varepsilon_k^2 -\frac{f^{(4)}(\xi_k^*)}{6}\varepsilon_k^3,\\
0 = f'(x^*) &= f'(x_k) - f''(x_k)\varepsilon_k + \frac{f^{(3)}(\xi_k')}{2}\varepsilon_k^2,\\
0 = f'(x^*) &= f'(x_k) - f''(\xi_k^\dagger)\varepsilon_k.
\end{align*}
Plugging these into \eqref{equ:iter2} and defining
\begin{align*}
A_k &\coloneqq f''(x_k) + \frac{f^{(3)}(x_k)}{2}\beta_k f'(x_k) + \frac{f^{(4)}(\xi_k)\beta_k^2f'(x_k)^2}{6},\\
B_k &\coloneqq \frac{f^{(4)}(\xi_k)}{6}\beta_k^2f''(\xi_k^\dagger)^2\varepsilon_k^3 - \frac{f^{(4)}(\xi_k^*)}{6}\varepsilon_k^3 - \frac{f^{(3)}(x_k)}{4}f^{(3)}(\xi_k')\beta_k\varepsilon_k^3,
\end{align*}
we obtain
\[
\varepsilon_{k+1} = \frac{\frac12 f^{(3)}(x_k)\varepsilon_k^2(f''(x_k)\beta_k+1) + B_k}{A_k}.
\]
Since $\beta_k = -(x_k - x_{k-1})/(f'(x_k) - f'(x_{k-1}))$, we may Taylor expand $f'(x_{k-1})$ at $x_k$ to get
\begin{align*}
f'(x_{k-1}) &= f'(x_k) + f''(x_k)(\varepsilon_{k - 1} - \varepsilon_k) +\frac{f^{(3)}(\xi_k^\ddagger)}{2}(\varepsilon_{k-1} - \varepsilon_k)^2
\end{align*}
for some $\xi_k^\ddagger$ between $x_{k-1}$ and $x_k$. Plugging it into
\[
\beta_k = -\frac{1}{f''(x_k)+\frac12 f^{(3)}(\xi_k^\ddagger)(\varepsilon_{k-1} - \varepsilon_k)}
\]
gives us
\[
\varepsilon_{k+1}
=\frac{\dfrac{f^{(3)}(x_k)f^{(3)}(\xi_k^\ddagger)\varepsilon_k^2(\varepsilon_{k-1} - \varepsilon_k)}{2(2f''(x_k)+f^{(3)}(\xi_k^\ddagger)(\varepsilon_{k-1} - \varepsilon_k))} + B_k}{A_k}.
\]
We deduce that
\[
\lim_{k\to\infty}\frac{|\varepsilon_{k}|}{|\varepsilon_{k-1}|} = 0, \qquad
\lim_{k\to\infty}\frac{|B_k|}{|\varepsilon_k^2\varepsilon_{k-1}|} = 0,
\]
and therefore
\[
\lim_{k\to\infty} \frac{|\varepsilon_{k+1}|}{|\varepsilon_k^2\varepsilon_{k-1}|} = \Bigl(\frac{f^{(3)}(x^*)}{2f^{(2)}(x^*)}\Bigr)^2.
\]
Hence the convergence order is $1 + \sqrt{2}$.
\end{proof}

The choice of $\beta_k$ above is exactly the Barzilai--Borwein (BB) step size for a univariate function \cite{Barzilai1988two}.  In the multivariate setting, $\beta_k$ will be replaced by the multivariate BB step size.
Theorem~\ref{superQ} provides the impetus for a first-order method with Steffensen updates and BB step size, namely, it is superquadratically convergent for univariate functions.
Such a high convergence order is clearly an overkill for a deterministic algorithm but our experiments in Section~\ref{sec:num} show that they are rewarding when the algorithm is randomized, as randomization inevitably compromises convergence speed. For easy comparison, we tabulate the convergence order, i.e., the largest $q$ such that $\lvert \varepsilon_{k+1} \rvert \le c\lvert \varepsilon_k \rvert^q$ for some $c >0$ and all $k$ sufficiently large, of various methods below:
\begin{center}
\begin{tabular}{l|c|c|l}
Method & Convergence & Derivatives & Steps\\\hline
Steepest descent & $1 $ & $1$st & single step \\
Secant = Barzilai--Borwein = quasi-Newton & $(1 + \sqrt{5})/2 $ & $1$st & mutltistep\\
Newton & $2 $ & $2$nd & single step \\
Steffensen & $2 $ & $1$st & single step \\
Steffensen--Barzilai--Borwein & $1 + \sqrt{2} $ & $1$st & multistep
\end{tabular}
\end{center}
Note that for a univariate function, Barzilai--Borwein step size and any quasi-Newton method with Broyden class updates (including  BFGS, DFP, SR1) reduce to the secant method. In particular, they are all two-step methods, i.e., its iterate at step $k$ depends on both $x_k$ and $x_{k-1}$. As a result Steffensen--Barzilai--Borwein method is also a two-step method as it involves the Brazlai--Borwein step size but Steffensen method is a one-step method.

\subsection{Deterministic multivariate setting}\label{sec:multvar}

There have been no shortage of proposals for extending Steffensen method to a multivariate or even infinite-dimensional setting \cite{Amat2016steffensen,
Ezquerro2014hybrid,Henrici,
Huang1970unified,
Johnson1968steffensen,
Nedzhibov2011approach,
Nievergelt1991Aitken,
Nievergelt1995condition,
Noda1981Aitken1,
Noda1986Aitken2,
Noda1986Aitken3,
Noda1990Aitken4,
Noda1992Aitken5}. However all of them rely on various multivariate versions of \emph{divided differences} that require evaluation and storage of $O(n^2)$ first derivatives in each step. Although they do avoid second derivatives, computationally they are just as expensive as Newton method and are unsuitable for modern large scale applications like training deep neural networks.

We will propose an alternative class of multivariate  Steffensen methods that use only $O(n)$ first derivatives, by emulating quasi-Newton methods \cite{Broyden1970The, fletcher1970new, Goldfarb1970A, Shanno1970Condition} and Barzilai--Borwein method \cite{Barzilai1988two} respectively. Our observation is that expensive multivariate divided differences can be completely avoided if we just use the ideas in Section~\ref{sec:uni} to \emph{define learning rates}. Another advantage is that these learning rates could be readily used in conjunction with existing stochastic optimization methods, as we will see in Section~\ref{sec:stoc}.

The key idea behind quasi-Newton method is the extension of univariate secant method to  a multivariate objective function $f:\mathbb{R}^d\rightarrow\mathbb{R}$ by replacing the finite difference approximation of $f''(x_k)$, i.e.,
$h_k = [f'(x_k) - f'(x_{k-1})]/(x_k - x_{k-1})$,
with the \emph{secant equation} $H_k s_k = y_k$ or
\begin{equation}\label{secant equation}
B_k y_k = s_k
\end{equation}
where $s_k = x_k - x_{k-1}$ and $y_k = \grad f(x_k) - \grad f(x_{k-1})$, avoiding the need to divide vectorial quantitites. Here $H_k$ (resp.\ $B_k$) is the approximate (resp.\ inverse) Hessian.

We use the same idea to extend Steffensen method to a multivariate setting, solving \eqref{secant equation} with
\[
s_k = \grad f(x_k), \qquad y_k = \grad f(x_k + \grad f(x_k)) - \grad f(x_k).
\]
Note that with these choices, \eqref{secant equation} roughly says that ``$B_k = s_k/y_k = \grad f(x_k)/[\grad f(x_k + \grad f(x_k)) - \grad f(x_k) ]$,'' which gives us $f'(x_k)/[f'(x_k + f'(x_k)) - f'(x_k)]$ as in the univariate Steffensen method when $d=1$ but is of course meaningless when $d > 1$. 
Nevertheless we may pick a minimum-norm solution to \eqref{secant equation}, which is easily seen to be given by the rank-one matrix
\[
B_k = \argmin_{B y_k = s_k}  \|B \|  = \frac{s_k y_k^\tp}{y_k^\tp y_k}
\]
regardless of whether $\lVert \, \cdot \, \rVert$ is the Frobenius or spectral norm. Hence we obtain a multivariate analogue of Steffensen method \eqref{uniSteff} as
\begin{equation}\label{Steffsy}
x_{k+1} = x_k - B_k \grad f(x_k) 
= x_k - \frac{[\grad f(x_k + \grad f(x_k)) - \grad f(x_k)]^\tp  \grad f(x_k)}{\lVert \grad f(x_k + \grad f(x_k)) - \grad f(x_k) \rVert^2} \grad f(x_k).
\end{equation}
We will call this \emph{quasi-Steffensen method} in analogy with quasi-Newton methods.

The key idea behind the Barzilai--Borwein method \cite{Barzilai1988two} is an alternative way of treating  the secant equation \eqref{secant equation}, whereby the approximate Hessian $B_k$ is assumed to take the form $B_k = \sigma_k I$ for some scalar $\sigma_k > 0$. Since in general it is not possible to find $\sigma_k$ so that \eqref{secant equation}  holds exactly with $B_k = \sigma_k I$, a best approximation is used instead. We seek $\sigma_k$ so that the residual of the secant equation $\|y_k - (1/\sigma_k) s_k\|^2$ or $\|\sigma_k y_k - s_k\|^2$ is minimized. The first minimization problem gives us
\begin{equation}\label{eq:bbmulti}
\sigma_k = \argmin_{\sigma>0} \; \|y_k - (1/\sigma) s_k\|^2 = \frac{s_k^\tp s_k}{s_k^\tp y_k} =\frac{\|\grad f(x_k)\|^2}{[\grad f(x_k + \grad f(x_k)) - \grad f(x_k)]^\tp \grad f(x_k)},
\end{equation}
and the second minimization gives the same expression as (\ref{Steffsy}). We will call the resulting iteration
\[
x_{k+1} = x_k - \frac{\|\grad f(x_k)\|^2}{[\grad f(x_k + \grad f(x_k)) - \grad f(x_k)]^\tp \grad f(x_k)} \grad f(x_k)
\]
\emph{Steffensen method} since it most resembles the univariate Steffensen method in  \eqref{uniSteff}. Note that the Barzilai--Borwein step size derived in \cite{Barzilai1988two} is
\begin{equation}\label{eq:bbori}
\beta_k = \frac{\|x_k - x_{k-1}\|^2}{[\grad f(x_k) - \grad f(x_{k-1})]^\tp(x_k - x_{k-1})}
\end{equation}
and differs significantly from \eqref{eq:bbmulti}. In particular, $x_{k+1} = x_k - \sigma_k \grad f(x_k)$  is a multistep method whereas $x_{k+1} = x_k - \beta_k \grad f(x_k)$ remains a single step method.

Both \eqref{Steffsy} and \eqref{eq:bbmulti} reduce to \eqref{uniSteff} when $f$ is univariate. Motivated by the univariate discussion before Theorem~\ref{superQ}, we combine features from \eqref{eq:bbmulti} and \eqref{eq:bbori} to obtain a \emph{Steffensen--Barzilai--Borwein method} in analogy with the univariate case \eqref{eq:bbs}:
\begin{equation}\label{eq:SBB}
x_{k+1} = x_k - \frac{\beta_k \|\grad f(x_k)\|^2}{[\grad f(x_k + \beta_k \grad f(x_k)) - \grad f(x_k)]^\tp \grad f(x_k)} \grad f(x_k).
\end{equation}
Note that \eqref{eq:SBB} reduces to \eqref{eq:bbs} when $f$ is univariate.
The stochastic version of \eqref{eq:SBB} will be our method of choice, supported by extensive empirical evidence some of which we will present in Section~\ref{sec:num}. 

In summary, we have four plausible learning rates.\label{pg:steps}
\begin{align*}
\eta^{\qS}_k &= \frac{[\grad f(x_k + \grad f(x_k)) - \grad f(x_k)]^\tp  \grad f(x_k)}{\lVert \grad f(x_k + \grad f(x_k)) - \grad f(x_k) \rVert^2}, \tag*{quasi-Steffensen:} \\
\eta^{\qSBB}_k &=  \frac{[\grad f(x_k + \beta_k \grad f(x_k)) - \grad f(x_k)]^\tp  \grad f(x_k)}{\lVert \grad f(x_k + \beta_k \grad f(x_k)) - \grad f(x_k) \rVert^2}, \tag*{quasi-Steffensen--Barzilai--Borwein:} \\
\eta^{\S}_k &= \frac{\|\grad f(x_k)\|^2}{[\grad f(x_k + \grad f(x_k)) - \grad f(x_k)]^\tp \grad f(x_k)},  \tag*{Steffensen:} \\
\eta^{\SBB}_k &= \frac{\beta_k \|\grad f(x_k)\|^2}{[\grad f(x_k + \beta_k \grad f(x_k)) - \grad f(x_k)]^\tp \grad f(x_k)}. \tag*{Steffensen--Barzilai--Borwein:}
\end{align*}
Here $\beta_k$ is the Barzilai--Borwein step size in \eqref{eq:bbori}. For a univariate function, the iterations with $\eta^{\qS}_k$ and $\eta^{\S}_k$ reduce to \eqref{uniSteff} whereas those with $\eta^{\qSBB}_k$ and $\eta^{\SBB}_k$ reduce to \eqref{eq:bbs}. The computational costs of all four learning rates are the same: two gradient evaluations and two inner products.

Note that our muiltivariate Steffensen and quasi-Steffensen methods are one-step methods --- $\eta^{\S}_k$ and $\eta^{\qS}_k$ depend only on $x_k$ --- just like the univariate Steffensen method. Steffensen--Barzilai--Borwein and quasi-Steffensen--Barzilai--Borwein are inevitably two-step methods because they involve the Barzilai--Borwein step size $\beta_k$, which has a two-step formula.

The main difference between our multivariate Steffensen methods and those in the literature  \cite{Amat2016steffensen,
Ezquerro2014hybrid,Henrici,
Huang1970unified,
Johnson1968steffensen,
Nedzhibov2011approach,
Nievergelt1991Aitken,
Nievergelt1995condition,
Noda1981Aitken1,
Noda1986Aitken2,
Noda1986Aitken3,
Noda1990Aitken4,
Noda1992Aitken5} is that ours are encapsulated as learning rates and avoid expensive multivariate divided differences. Recall that for $g = (g_1,\dots,g_n) : \mathbb{R}^n \to \mathbb{R}^n$, its divided difference \cite{Potra} at $x,y\in \mathbb{R}^n$ is the matrix $\llbracket x,y \rrbracket \in \mathbb{R}^{n \times n}$ whose $(i,j)$th entry is
\[
\llbracket x,y \rrbracket_{ij} \coloneqq
\begin{cases}
\dfrac{g_i(x_1,\dots,x_j,y_{j+1},\dots,y_n)-g_i(x_1,\dots,x_{j-1},y_{j},\dots,y_n)}{x_j-y_j} &x_j \neq y_j,\\
\dfrac{\partial g_i}{\partial x_j} (x_1,\dots,x_j,y_{j+1},\dots,y_n) &x_j = y_j,
\end{cases}
\]
for $i,j =1,\dots,n$.

In a stochastic setting, the learning rates $\eta^{\S}_k,\eta^{\qS}_k,\eta^{\SBB}_k,\eta^{\qSBB}_k$ share the same upper and lower bounds in Lemma~\ref{etabound:SSM} and as a result, the linear convergence conclusion in Theorem~\ref{linconv:SSBB} applies alike to all four of them. Our experiments also indicate that $\eta^{\qS}_k$ and $\eta^{\S}_k $ have similar performance and likewise for $\eta^{\qSBB}_k$ and $\eta^{\SBB}_k $, although there is a slight difference between $\eta^{\S}_k $ and $\eta^{\SBB}_k$.   One conceivable advantage of the `quasi' variants is that for a given $\grad f(x_k)$, the denominator vanishes only at a single point, e.g., when $\grad f(x_k +\grad f(x_k)) = \grad f(x_k)$, as opposed to a whole hyperplane, e.g., whenever $\grad f(x_k +  \grad f(x_k)) - \grad f(x_k) \perp \grad f(x_k)$. Nevertheless, in all our experiments on their stochastic variants, this has never been an issue.

We prefer the slightly simpler expressions of the Steffensen and Steffensen--Barzilai--Borwein methods and will focus our subsequent discussions on them. Their `quasi' variants may be taken as nearly equivalent alternatives for users who may have some other reasons to favor them.

\subsection{Stochastic multivariate setting}\label{sec:stoc}

Encapsulating Steffensen method in the form of learning rates offers an additional advantage --- it is straightforward to incorporate them into many stochastic optimization algorithms, which we will do next.

Standard gradient descent applied to \eqref{large sum}  requires the evaluation of $n$ gradients. The  stochastic gradient descent (SGD), instead of using the full gradient $\grad f(x_k)$, relies on an unbiased estimator $g_k$ with $\mathbb{E}[g_k] = \grad f(x_k)$ \cite{Robbins1951sto}.
One common randomization is to draw $i_k \in \{1, \dots, n\}$ randomly and set $g_k = \grad f_{i_{k}}(x_k)$, giving the update:
\[
x_{k+1} = x_k-\eta_k \grad f_{i_{k}}(x_k).
\]
Note that $\mathbb{E}[\grad f_{i_k}(x_k) \mid x_k] = \grad f(x_k)$ and its obvious advantage is that each step relies only  on a single gradient $\grad f_{i_k}$, resulting in a computational cost that is $1/n$ that of the standard gradient descent. While we could adopt this procedure to randomize our Steffensen and Steffensen--Barzilai--Borwein iterations, we will use a slightly more sophisticated variant with \emph{variance reduction} and \emph{minibatching}.

The price of randomization is paid in the form of variance, as the stochastic gradient $\grad f_{i_k}(x_k)$ equals the gradient $\grad f(x_k)$ only in expectation but each $\grad f_{i_k}(x_k)$ is different. Of the many variance reduction strategies, one of the best known and simplest is the stochastic variance reduced gradient method (SVRG)  \cite{johnson2013accelerating}, based on the tried-and-tested notion of control variates in Monte Carlo methods. We will emulate SVRG to randomize \eqref{Steffsy} and \eqref{eq:SBB}.  

The basic idea of SVRG is to compute the full gradient once every $m$ iterations for some fixed $m$ and use it to generate stochastic gradients with lower variance in the next $m$ iterations:
\[
x_{k+1} = x_k - \eta_k\bigl(\grad f_{i_k}(x_k) - \grad f_{i_k}(\widetilde{x}) + \grad f(\widetilde{x})\bigr).
\]
Here $\widetilde{x}$ denotes the point where full gradient is computed. Notice that when $k\to \infty$, $x_k$ and $\widetilde{x}$ are very close to the optimal point $x^*$. As $x_k$ and $\widetilde{x}$ are highly correlated, the variability of the stochastic gradient is reduced as a result \cite{johnson2013accelerating}.

We may similarly randomize multivariate Steffensen method. Our stochastic Steffensen method (SSM) in Algorithm~\ref{alg:SSM} operates in two nested loops. In the $k$th iteration of the outer loop, we compute two full gradients $\grad f(x_k)$ and $\grad f(x_k + \grad f(x_k))$. Note that $x_k$ plays the role of $\widetilde{x}$ in the above paragraph. These two terms are used for computing the Steffensen learning rate:
\begin{equation}\label{eq:etam}
\eta^{\SS}_{k} = \frac{1}{\sqrt{m}} \cdot\frac{\|\grad f(x_k)\|^2}{[\grad f(x_k + \grad f(x_k)) - \grad f(x_k)]^\tp \grad f(x_k)}.
\end{equation}
In the $(t+1)$th iteration of the inner loop, we use $\grad f(x_k)$ to generate the stochastic gradient with lower variance
\[
        v_{k,t} = \grad f_{i_t}(x_{k,t}) - \grad f_{i_t}(x_k) + \grad f(x_k),
\]
with $i_t \in \{1,\dots,n\}$ sampled uniformly. The updating rule takes the form
\[
x_{k,t+1} = x_{k,t} - \eta^{\SS}_{k} v_{k,t}
\]
where the search direction is known as the \emph{variance-reduced stochastic gradient}. Note that the learning rate $\eta_k$ given by \eqref{eq:etam} has an extra $1/\sqrt{m}$ factor; we will see how $m$ should be chosen in Section~\ref{sec:conv}.

\begin{algorithm}[htbp]
   \caption{Stochastic Steffensen Method (SSM)}
   \label{alg:SSM}
\begin{algorithmic}[1]
   \STATE {\bfseries Input:} initial state $x_0$, inner loop size $m$, data size $n$.
   \FOR{$k=0,1,\dots$}
   \STATE Compute full gradients $\grad f(x_k)$ and $\grad f(x_k + \grad f(x_k))$.
   \STATE Compute stochastic Steffensen learning rate
     \[
     \eta^{\SS}_{k} = \frac{1}{\sqrt{m}} \cdot\frac{\|\grad f(x_k)\|^2}{[\grad f(x_k + \grad f(x_k)) - \grad f(x_k)]^\tp \grad f(x_k)}.
     \]
    \STATE Set $x_{k,0} = x_k$.
    \FOR{$t=0$ {\bfseries to} $m-1$}
     \STATE Sample  $i_t \in \{1,\dots,n\}$ uniformly.
     \STATE Compute variance-reduced stochastic gradient
\[
             v_{k,t} = \grad f_{i_t}(x_{k,t}) - \grad f_{i_t}(x_k) + \grad f(x_k).
\]
     \STATE Update $x_{k,t+1} = x_{k,t} - \eta^{\SS}_{k}v_{k,t}$.
    \ENDFOR
  \STATE Set $x_{k+1} = x_{k,i}$ for uniformly chosen $i \in \{0, \dots, m - 1\}$.
   \ENDFOR
\end{algorithmic}
\end{algorithm}

Aside from variance reduction, we include another common enhancement called \emph{minibatching}. Minibatched SGD is a trade-off between SGD and gradient descent (GD) where the cost function (and therefore its gradient) is averaged over a small number of samples. SGD has a batch size of one whereas GD has a batch size that includes all training samples. In each iteration, we sample a minibatch $S_{k}\subseteq\{1,\dots,n\}$ with $|S_{k}|=b$ a small number and update
\[
x_{k+1} = x_k - \eta_k\frac{1}{|S_{k}|}\sum_{j\in S_{k}}\grad f_{j}(x_k) \eqqcolon x_k - \eta_k \grad f_{S_k}(x_k).
\]
Minibatched SGD smooths out some of the noise in SGD but maintains the ability to escape local minima.  The minibatch size $b$ is kept small, thus preserving the cost-saving benefits of SGD.  We want a small $b$ to minimize gradient computations and a large $m$ so that full gradients are computed only after a large number of iterations. With these considerations, we replace the factor of $1/\sqrt{m}$ in \eqref{eq:etam} by $b / m$. In Section~\ref{sec:conv}, we will see that this choice also allows us to establish linear convergence.

Upon incorporating (i) a Barzilai--Borwein step size, (ii) variance reduction, and (iii) minibatching, we arrive at the stochastic Steffensen--Barzilai--Borwein method (SSBB) in Algorithm~\ref{alg:SSBB}. This is our method of choice in this article.
\begin{algorithm}[htbp]
   \caption{Stochastic Steffensen--Barzilai--Borwein Method (SSBB)}
   \label{alg:SSBB}
\begin{algorithmic}[1]
    \STATE {\bfseries Input:} initial state $x_0$, inner loop size $m$, minibatch size $b$, data size $n$, Barzilai-Borwein learning rate $\beta_0 = -1$.
    \FOR{$k=0,1,\dots$}
        \STATE Compute  full gradient $\grad f(x_k)$.
        \IF{$k > 0$}
            \STATE Set $s_{k} = x_{k} - x_{k-1}$ and $y_{k} = \grad f(x_k) - \grad f(x_{k-1})$.
            \STATE Compute Barzilai--Borwein learning rate
	        \[
	        \beta_{k} = -\frac{\|s_{k}\|^2}{s_{k}^\tp y_{k}}.
            \]
        \ENDIF
        \STATE Compute the stochastic Steffensen--Barzilai--Borwein learning rate
        \[
        \eta^{\SSBB}_{k} = \frac{b}{m}\cdot\frac{\beta_k\|\grad f(x_k)\|^2}{[\grad f(x_k + \beta_k \grad f(x_k)) - \grad f(x_k)]^\tp \grad f(x_k)}.
        \]
        \STATE Set $x_{k,0} = x_k$.
        \FOR{$t=0$ {\bfseries to} $m-1$}
            \STATE Sample minibatch $S_{k,t}\subseteq\{1,\dots,n\}$ uniformly with $|S_{k,t}| = b$.
            \STATE Compute variance-reduced stochastic gradient
\[
                    v_{k,t} = \grad f_{S_{k,t}}(x_{k,t}) - \grad f_{S_{k,t}}(x_{k,t}) + \grad f(x_{k,t}).
\]
            \STATE Update $x_{k,t+1} = x_{k,t} - \eta^{\SSBB}_{k}v_{k,t}$.
        \ENDFOR
        \STATE Set $x_{k+1} = x_{k,i}$ for uniformly chosen $i \in \{0, \dots, m - 1\}$.
    \ENDFOR
\end{algorithmic}
\end{algorithm}

Although we did not include minibatching in  Algorithm~\ref{alg:SSM}'s pseudocode to avoid clutter, we will henceforth assume that it is also minibatched.
The randomization, variance reduction, and minibatching all apply verbatim when the learning rates in  Algorithms~\ref{alg:SSM} and \ref{alg:SSBB} are replaced respectively by the quasi-Steffensen and  quasi-Steffensen--Barzilai--Borwein learning rates on p.~\pageref{pg:steps}. Nevertheless, as we have mentioned, our numerical experiments do not show that the resulting algorithms differ in performance from that of Algorithms~\ref{alg:SSM} and \ref{alg:SSBB}.

\subsection{Randomized Kaczmarz method as a special case}\label{sec:kacz}

Given $A\in \mathbb{R}^{m \times n}$ of full row rank with \emph{row} vectors $a_1,\dots,a_m \in \mathbb{R}^n$ and $b\in \mathbb{R}^m$ in the image of $A$, the Kaczmarz method \cite{Kaczmarz,KaczEng}  solves the consistent linear system $Ax=b$ via
\[
x_{k+1} = x_k+ \frac{b_i- a_i^\tp x_k}{\lVert a_{i} \rVert^2} a_i,
\] 
with $i = k \mod m$, $i =1,\dots,  m$. The iterative method has remained relatively obscure, almost unheard of in numerical linear algebra,  until it was randomized in  \cite{strohmer2009randomized}, which essentially does
\[
x_{k+1} = x_k+ \frac{b_{i_k} - a_{i_k}^\tp x_k}{\lVert a_{i_k} \rVert^2} a_{i_k},
\]
where $i_k \in \{1,\dots,m\}$ is now \emph{sampled} with probability $\lVert a_{i_k}\rVert^2/\lVert A \rVert^2$.

We will see that randomized Kaczmarz method is equivalent to applying stochastic Steffensen method, with or without Barzilai--Borwein step size, to minimize the quadratic function $f : \mathbb{R}^n \to \mathbb{R}$,
\[
f(x) \coloneqq  \frac{1}{2}\sum_{i = 1}^m f_i(x) = \frac{1}{2}\sum_{i = 1}^m (a_i^\tp x - b_i)^2.
\]
While it is sometimes claimed that SGD has this property, this is not quite true. Suppose $i_k \in \{1,\dots,m\}$ is the random row index sampled at the $k$th step, the update rule in SGD gives
\[
x_{k+1} = x_k - \eta_k  (a_{i_k}^\tp x_k - b_i) a_{i_k},
\]
and the update rule in SLBFGS is even further from this.
So one needs to impose further assumptions \cite{Needell} on the learning rate to get randomized Kaczmarz method, which requires that $\eta_k = 1/\|a_{i_k}^2\|$. If we use the Steffensen method, we get from \eqref{eq:SBB} that
\[
\eta^{\S}_k = \frac{\|\grad f_{i_k}(x_k)\|^2}{[\grad f_{i_k}\bigl(x_k + \grad f_{i_k}(x_k)\bigr) - \grad f_{i_k}(x_k)]^\tp \grad f_{i_k}(x_k)} = \frac{1}{\|a_{i_k}\|^2};
\]
and using Steffensen--Barzilai--Borwein method makes no difference:
\[
\eta^{\SBB}_k = \frac{\beta_k\|\grad f_{i_k}(x_k)\|^2}{[\grad f_{i_k}\bigl(x_k + \beta_k\grad f_{i_k}(x_k)\bigr) - \grad f_{i_k}(x_k)]^\tp \grad f_{i_k}(x_k)} = \frac{1}{\|a_{i_k}\|^2},
\]
as $\beta_k = \|x_k - x_{k-1}\|^2/(x_k - x_{k-1})^\tp[\grad f_{i_k}(x_k) - \grad f_{i_k}(x_{k-1})] = 1/\|a_{i_k}\|^2$.

\section{Convergence Analysis}\label{sec:conv}

In this section, we establish the linear convergence of our stochastic Steffensen methods Algorithm~\ref{alg:SSM} (SSM) and Algorithm~\ref{alg:SSBB} (SSBB) for solving \eqref{large sum} under standard assumptions. 
We would like to stress that these convergence results are intended to provide a minimal  theoretical guarantee and do not really do justice to the actual performance of SSBB. The experiments in Section~\ref{sec:num} indicate that the convergence of SSBB is often superior to other existing methods like SGD and SVRG, with or without Barzilai--Borwein step size, or even SLBFGS.
However, we are unable to prove this theoretically, only that it is linearly convergent like the other methods.

For easy reference, we reproduce the minibatched SVRG algorithm in \cite[Algorithm~1]{babanezhad2015stopwasting} as Algorithm~\ref{alg:minibatchSVRG}.
\begin{algorithm}[htb]
   \caption{Minibatched SVRG}
   \label{alg:minibatchSVRG}
\begin{algorithmic}[1]
   \STATE {\bfseries Input:} initial state $x_0$, inner loop size $m$, minibatch size $b$, data size $n$.
   \FOR{$k=0,1,\dots$}
    \STATE Compute  full gradient $\grad f(x_k)$.
    \STATE Set $x_{k,0} = x_k$.
    \FOR{$t=0$ {\bfseries to} $m-1$}
     \STATE Sample minibatch $S_{k,t}\subseteq\{1,\dots,n\}$ uniformly with $|S_{k,t}| = b$.
     \STATE Compute variance-reduced stochastic gradient
     \[
     v_{k,t} = \grad f_{S_{k,t}}(x_{k,t}) - \grad f_{S_{k,t}}(x_k) + \grad f(x_k).
     \]
     \STATE Update $x_{k,t+1} = x_{k,t} - \eta_k v_{k,t}$.
    \ENDFOR
  \STATE Set $x_{k+1} = x_{k,i}$ for uniformly chosen $i \in \{0, \dots, m - 1\}$.
   \ENDFOR
\end{algorithmic}
\end{algorithm}
We need to establish the linear convergence of Algorithm~\ref{alg:minibatchSVRG} for our own convergence results in Sections~\ref{convSSM} and \ref{convSSBB} but we are unable to find such a result in the literature. In particular, the convergence results in \cite[Propositions~2--4]{babanezhad2015stopwasting} and \cite[Theorem~1]{xiao2014proximal} are for more sophisticated variants of Algorithm~\ref{alg:minibatchSVRG}. So we will provide a version following the same line of arguments in \cite[Theorem~1]{xiao2014proximal} but tailored to our own requirements.

Our linear convergence proofs for SSM and SSBB are a combination of the proofs in \cite{nitanda2016accelerated,xiao2014proximal} adapted for our purpose. In particular, we  quote \cite[Lemma~A]{nitanda2016accelerated} and prove a simplied version of \cite[Lemma~3]{xiao2014proximal} for easy reference.
\begin{lemma}[Nitanda]\label{Combination}
Let $\xi_1,\dots,\xi_n \in \mathbb{R}^{d}$ and $\bar{\xi} \coloneqq \sum_{i=1}^n\xi_i$. Let $S$ be a $b$-element subset chosen uniform randomly from all $b$-element subsets of $\{1,2, \ldots, n\}$. Then
\[
\mathbb{E}_{S}\Bigl\|\frac{1}{b} \sum\nolimits_{i \in S} \xi_{i}-\bar{\xi}\Bigr\|^2=\frac{n-b}{b(n-1)} \mathbb{E}_{i}\bigl\|\xi_{i}-\bar{\xi}\bigr\|^2.
\]
\end{lemma}

For the rest of this section, we will need to assume, as is customary in such proofs of linear convergence, that our objective $f$ is $\mu$-strongly convex and the gradient of each additive component $f_i$ is $L$-Lipschitz continuous.  It follows that $\grad f$ must also be $L$-Lipschitz continuous.
\begin{assumption}\label{ass}
Assume that the  function $f$ in \eqref{large sum} satisfies
\begin{gather*}
f(w) \ge f(v)+\grad f(v)^\tp (w-v)+\frac{\mu}{2}\|v-w\|^2,\\
\|\grad f_i(v)-\grad f_i(w)\|  \le L\|v-w\|
\end{gather*}
for any $v, w \in \mathbb{R}^{d}$, $i =1,\dots,n$.
\end{assumption}

Applying Lemma~\ref{Combination} with $\xi_i = v_i^{k,t}$ and \cite[Corollary~3]{xiao2014proximal}, we may bound the variance of a minibatched variance-reduced gradient as follows.
\begin{lemma}\label{stobound}
Let $f$ be as in Assumption~\ref{ass} with $x^*\coloneqq\argmin_{x} f(x)$. Let
\[
 v_i^{k,t} = \grad f_i(x_{k,t}) - \grad f_i(x_k) + \grad f(x_k), \qquad v_{k,t} = \frac{1}{b}\sum_{i\in S_{k,t}} v_{i}^{k,t}.
\]
Then
\[
\mathbb{E}\|v_{k, t}-\grad f(x_{k,t})\|^2 \le \frac{4 L}{b} \bigl[f(x_{k, t})-f(x^*)+f(x_k)-f(x^*)\bigr].
\]
\end{lemma}

The next lemma, a simplified version of \cite[Lemma~3]{xiao2014proximal}, gives a lower bound of the optimal value $f(x^*)$ useful in our proof of linear convergence.
\begin{lemma}\label{restbound}
Let $\Delta_{k,t} \coloneqq v_{k,t} - \grad f(x_{k,t})$ and $\eta_k$ be a learning rate with  $0<\eta_k \le 1 / L$. Then with the same assumptions and notations in Lemma~\ref{stobound}, we have
\[
f(x^*)\ge f(x_{k,t+1})+v_{k,t}^{\tp}(x^*-x_{k,t})+\frac{\eta_k}{2}\|v_{k,t}\|^2+\frac{\mu}{2}\|x^*-x_{k,t}\|^2+\Delta_{k,t}^{\tp}(x_{k,t+1}-x^*).
\]
\end{lemma}

\begin{proof}
By the strong convexity of $f$, we have
\[
f(x^*) \ge f(x_{k,t}) + \grad f(x_{k,t})^{\tp}(x^* - x_{k,t}) + \frac{\mu}{2}\|x^* - x_{k,t}\|^2.
\]
By the smoothness of $f$, we have
\[
f(x_{k,t}) \ge f(x_{k,t+1}) - \grad f(x_{k,t+1})^{\tp}(x_{k,t+1} - x_{k,t}) - \frac{L}{2}\|x_{k,t+1} - x_{k,t}\|^2.
\]
Summing the two inequalities, we get
\[
f(x^*)\ge f(x_{k,t+1}) + \grad f(x_{k,t})^{\tp}(x^* - x_{k,t+1}) + \frac{\mu}{2}\|x^* - x_{k,t}\|^2 - \frac{L\eta_k^2}{2}\|v_{k,t}\|^2.
\]
The second term on the right simplifies as
\begin{align*}
\grad f(x_{k,t})^{\tp}(x^* - x_{k,t+1}) &= \grad f(x_{k,t})^{\tp}(x^* - x_{k,t+1}) + (v_{k,t} - v_{k,t})^{\tp}(x^* - x_{k,t+1})\\
& = v_{k,t}^{\tp}(x^* - x_{k,t+1}) + (v_{k,t} - \grad f(x_{k,t}))^{\tp}(x_{k,t+1} - x^*)\\
& = v_{k,t}^{\tp}(x^* - x_{k,t+1}) + \eta_k\|v_{k,t}\|^2.
\end{align*}
If the learning rate satisfies $0 < \eta_k \le 1/L$, then
\begin{align*}
f(x^*)& \ge f(x_{k,t+1})+v_{k,t}^{\tp}(x^*-x_{k,t})+\frac{\eta_k}{2}(2-L\eta_k)\left\|v_{k,t}\right\|^2+\frac{\mu}{2}\left\|x^*-x_{k,t}\right\|^2+\Delta_{k,t}^{\tp}(x_{k,t+1}-x^*)\\
& \ge f(x_{k,t+1})+v_{k,t}^{\tp}(x^*-x_{k,t})+\frac{\eta_k}{2}\left\|v_{k,t}\right\|^2+\frac{\mu}{2}\left\|x^*-x_{k,t}\right\|^2+\Delta_{k,t}^{\tp}(x_{k,t+1}-x^*),
\end{align*}
as required.
\end{proof}

\begin{theorem}[Linear convergence of Algorithm~\ref{alg:minibatchSVRG}]\label{maintheorem}
Let $f$ be as in Assumption~\ref{ass} with $x^*\coloneqq \argmin_{x} f(x)$. For the $(k+1)$th iteration of outer loop in Algorithm~\ref{alg:minibatchSVRG},
\[
\mathbb{E}[f(x_{k+1}) - f(x^*)] \le \left[\frac{b}{m\mu\eta_k(b - 4L\eta_k)} + \frac{4(m+1)L\eta_k}{m(b-4L\eta_k)}\right] [f(x_k) - f(x^*)].
\]
If $m$, $\eta_k$, and $b$ are chosen so that
\[
\rho_k = \frac{b}{m\mu\eta_k(b - 4L\eta_k)} + \frac{4(m+1)L\eta_k}{m(b-4L\eta_k)} \le \rho < 1,\qquad \eta_k < \min\Bigl( \frac{b}{4L}, \frac{1}{L} \Bigr),
\]
then Algorithm~\ref{alg:minibatchSVRG} converges linearly in expectation with
\[
\mathbb{E}[f(x_k) - f(x^*)]\le \rho^k[f(x_0) - f(x^*)].
\]
\end{theorem}

\begin{proof}
For the iteration in the inner loop, we apply Lemma~\ref{restbound} to get
\begin{equation}\label{distancebound}
\begin{aligned}
\left\|x_{k,t+1}-x^*\right\|^2 &= \left\|x_{k,t} - x^*\right\|^2 - 2\eta_kv_{k,t}^{\tp}(x_{k,t} - x^*) + \eta_k^2\left\|v_{k,t}\right\|^2 \\
& \le \left\|x_{k,t} - x^*\right\|^2 + 2\eta_k[f(x^*) - f(x_{k,t+1})] - 2\eta_k\Delta_{k,t}^{\tp}(x_{k,t+1} - x^*). 
\end{aligned}
\end{equation}
Lemma~\ref{restbound} requires that the learning rate $\eta_k\leq 1/L$. Let $\bar{x}_{k,t+1} \coloneqq x_{k,t}-\eta_k\nabla f(x_{k,t})$. Then the last term in \eqref{distancebound} may be written as
\begin{align*}
-2\eta_k\Delta_{k,t}^{\tp}(x_{k,t+1} - x^*) &= -2\eta_k\Delta_{k,t}^{\tp}(x_{k,t+1} - \bar{x}_{k,t+1}) - 2\eta_k\Delta_{k,t}^{\tp}(\bar{x}_{k,t+1} - x^*)\\
& = 2\eta_k^2\left\|\Delta_{k,t}\right\|^2 - 2\eta_k\Delta_{k,t}^{\tp}(\bar{x}_{k,t+1} - x^*).
\end{align*}
Plugging this into \eqref{distancebound} and taking expectations on both sides conditioned on $x_{k,t}$ and $x_k$ respectively, we get
\begin{align*}
\mathbb{E}\|x_{k,t+1}-x^*\|^2&\le \|x_{k,t} - x^*\|^2 + 2\eta_k[\eta_k\mathbb{E}\|\Delta_{k,t}\|^2 - \mathbb{E}[\Delta_{k,t}^{\tp}(\bar{x}_{k,t+1} - x^*)] - (f(x_{k,t+1}) - f(x^*))]\\
&=  \|x_{k,t} - x^*\|^2 + 2\eta_k[\eta_k\mathbb{E}\|\Delta_{k,t}\|^2 - (f(x_{k,t+1}) - f(x^*))],
\end{align*}
where the last equality follows from $\mathbb{E}[\Delta_{k,t}] = 0$. Set $\gamma \coloneqq 8L\eta_k^2/b$. By Lemma~\ref{stobound}, we have
\[
\mathbb{E}\|x_{k,t+1}-x^*\|^2
\le \|x_{k,t} - x^*\|^2 + \gamma[f(x_{k,t}) - f(x^*) + f(x_k) - f(x^*)] - 2\eta_k\mathbb{E}[f(x_{k,t+1}) - f(x^*)].
\]
For $t=0,\dots,m-1$, we have
\[
\mathbb{E}\left\|x_{k,t+1}-x^*\right\|^2 + 2\eta_k\mathbb{E}[f(x_{k,t+1}) - f(x^*)]
\le \|x_{k,t} - x^*\|^2 + \gamma[f(x_{k,t}) - f(x^*) + f(x_k) - f(x^*)].
\]
Summing this inequality over all $t=0,\dots,m-1$, the left hand side becomes
\[
\text{LHS} = \sum_{t=0}^{m-1}\mathbb{E}\left\|x_{k,t+1}-x^*\right\|^2 + 2\eta_k\sum_{t=0}^{m-1}\mathbb{E}[f(x_{k,t+1}) - f(x^*)],
\]
and the right hand side becomes
\[
\text{RHS} = \sum_{t=0}^{m-1}\|x_{k,t} - x^*\|^2 + \gamma\sum_{t=0}^{m-1}\mathbb{E}[f(x_{k,t}) - f(x^*)] + \gamma m\mathbb{E}[f(x_k) - f(x^*)].
\]
By the definition of $x_{k+1}$ in Algorithm~\ref{alg:minibatchSVRG},
\[
\mathbb{E}[f(x_{k+1})] = \frac{1}{m}\sum_{t=1}^mf(x_{k,t}),
\]
and so, bearing in mind that  $\text{LHS} \le \text{RHS}$,
\begin{align*}
\mathbb{E}\|x_{k,m}&-x^*\|^2 + 2\eta_km\mathbb{E}[f(x_{k+1}) - f(x^*)]\\
&\le  \mathbb{E}\left\|x_{k,0}-x^*\right\|^2 + \gamma m\mathbb{E}[f(x_k) - f(x^*)] + \gamma\sum_{t=0}^{m-1}\mathbb{E}[f(x_{k,t}) - f(x^*)]\\
&\le  \mathbb{E}\left\|x_{k,0}-x^*\right\|^2 + \gamma m\mathbb{E}[f(x_k) - f(x^*)] + \gamma\sum_{t=0}^{m}\mathbb{E}[f(x_{k,t}) - f(x^*)]\\
&=\mathbb{E}\left\|x_{k,0}-x^*\right\|^2 + \gamma m\mathbb{E}[f(x_k) - f(x^*)] + \gamma m\mathbb{E}[f(x_{k+1}) - f(x^*)] + \gamma[f(x_k) - f(x^*)].
\end{align*}
Hence we have
\begin{align*}
2\eta_km\mathbb{E}[f(x_{k+1}) - f(x^*)] \le\ & 2\eta_km\mathbb{E}[f(x_{k+1}) - f(x^*)] + \mathbb{E}\left\|x_{k,m}-x^*\right\|^2\\
\le\ & \mathbb{E}\left\|x_k-x^*\right\|^2 + \gamma(m+1)\mathbb{E}[f(x_k) - f(x^*)] + \gamma m\mathbb{E}[f(x_{k+1}) - f(x^*)].
\end{align*}
Rearranging terms and applying strong convexity of $f$, we have
\begin{align*}
\biggl(2\eta_k - \frac{8L\eta_k^2}{b} \biggr)m\mathbb{E}[f(x_{k+1}) - f(x^*)] & \le  \mathbb{E}\left\|x_k-x^*\right\|^2 + \frac{8(m+1)L\eta_k^2}{b}\mathbb{E}[f(x_k) - f(x^*)]\\
&\le  \frac{2}{\mu}[f(x_k) - f(x^*)] + \frac{8(m+1)L\eta_k^2}{b}\mathbb{E}[f(x_k) - f(x^*)].
\end{align*}
Here we require that $2\eta_k > 8L\eta_k^2 / b$ and thus $\eta_k < b / (4L)$, leading to
\[
\mathbb{E}[f(x_{k+1}) - f(x^*)] \le \rho_k [f(x_k) - f(x^*)]
\]
with
\[
\rho_k \coloneqq \frac{b}{m\mu\eta_k(b - 4L\eta_k)} + \frac{4(m+1)L\eta_k}{m(b-4L\eta_k)}.
\]
Choose $m, \eta_k$ so that $\rho_k \le \rho < 1$ and apply the last inequality recursively, we get
\[
\mathbb{E}[f(x_k) - f(x^*)]\le  \rho^k[f(x_0) - f(x^*)]
\]
as required.
\end{proof}

\subsection{Linear convergence of stochastic Steffensen method}\label{convSSM}

With Theorem~\ref{maintheorem}, we may deduce the linear convergence of  Algorithm~\ref{alg:SSM} as a special case of Algorithm~\ref{alg:minibatchSVRG} with $b = 1$ (no minibatching) and 
$\eta_k = \eta^{\SS}_k$ (SSM learning rate).
\begin{lemma}\label{etabound:SSM}
Let $f$ be as in Assumption~\ref{ass}. Then the stochastic Steffensen learning rate
\[
\eta^{\SS}_k = \frac{1}{\sqrt{m}}\cdot\frac{\|\grad f(x_k)\|^2}{{\grad f(x_k)}^\tp (\grad f(x_k + \grad f(x_k)) - \grad f(x_k))}
\]
satisfies
\[
\frac{1}{\sqrt{m}L} \leq \eta^{\SS}_k \le \frac{1}{\sqrt{m}\mu}.
\]
\end{lemma}
\begin{proof}
Since $\grad f$ is $L$-Lipschitz, a lower bound is given by
\[
\eta^{\SS}_k \ge \frac{1}{\sqrt{m}}\cdot\frac{\|\grad f(x_k)\|^2}{L\|\grad f(x_k)\|^2} = \frac{1}{\sqrt{m}L}.
\]
The required upper bound follows the $\mu$-strong convexity of $f$.
\end{proof}

\begin{corollary}[Linear convergence of SSM]
\label{linconv:SSM}
Let $f$ be as in Assumption~\ref{ass} with $x^*\coloneqq \argmin_{x} f(x)$. If $m, \eta$ is chosen so that
\[
\rho = \frac{\kappa + 4(1+1/m)\kappa}{\sqrt{m} -4 \kappa} < 1,
\]
where $\kappa = L / \mu$ is the condition number, then Algorithm~\ref{alg:SSM} converges linearly in expectation with
\[
\mathbb{E}[f(x_k) - f(x^*)]\le \rho^k[f(x_0) - f(x^*)].
\]
\end{corollary}

\begin{proof}
By Theorem~\ref{maintheorem}, we have
\[
\mathbb{E}[f(x_{k+1}) - f(x^*)] \le \left[\frac{1}{m\mu\eta^{\SS}_k(1 - 4L\eta^{\SS}_k)} + \frac{4(m+1)L\eta^{\SS}_k}{m(1-4L\eta^{\SS}_k)}\right] \mathbb{E}[f(x_k) - f(x^*)]
\]
as long as $\eta^{\SS}_k < 1/(4L)$. Lemma~\ref{etabound:SSM} shows that this holds for $m > 16\kappa^2$. The upper and lower bounds in Lemma~\ref{etabound:SSM} also give
\begin{align*}
    \rho^{\SS}_k &= \frac{1}{m\mu\eta^{\SS}_k(1 - 4L\eta^{\SS}_k)} + \frac{4(m+1)L\eta^{\SS}_k}{m(1-4L\eta^{\SS}_k)}\\
    &\leq \frac{1}{m\mu\frac{1}{\sqrt{m}L}(1-4L\frac{1}{\sqrt{m}\mu})} + \frac{4(m + 1)L\frac{1}{\sqrt{m}\mu}}{m(1-4L\frac{1}{\sqrt{m}\mu})}\\
    &= \frac{\kappa + 4(1+1/m)\kappa}{\sqrt{m} -4 \kappa}.
\end{align*}
Hence if $m$ is chosen so that $\rho < 1$, we have
\[
\mathbb{E}[f(x_k) - f(x^*)]\le \rho^k[f(x_0) - f(x^*)]
\]
as required.
\end{proof}

\subsection{Linear convergence of stochastic Steffensen--Barzilai--Borwein}\label{convSSBB}

The linear convergence of  Algorithm~ \ref{alg:SSBB} likewise follows from Theorem~\ref{maintheorem} with $\eta_k = \eta^{\SSBB}_k$.
\begin{lemma}
\label{etabound:SSBB}
Let $f$ be as in Assumption~\ref{ass}. Then the stochastic Steffensen--Barzilai--Borwein learning rate
\[
\eta^{\SSBB}_k = \frac{b}{m}\cdot\frac{\beta_k\|\grad f(x_k)\|^2}{[\grad f(x_k + \beta_k\nabla f(x_k)) - \grad f(x_k)]^\tp \grad f(x_k)}
\]
satisfies
\[
\frac{b}{mL} \le \eta_k^{\SSBB} \le \frac{b}{m\mu}.
\]
\end{lemma}
\begin{proof}
Similar to that of Lemma~\ref{etabound:SSM}.
\end{proof}

\begin{corollary}[Linear convergence of SSBB]\label{linconv:SSBB}
Let $f$ be as in Assumption~\ref{ass} with $x^*\coloneqq \argmin_{x} f(x)$. If $m$ and $b$ are chosen so that
\[
\rho = \frac{\kappa m + 4\kappa b(1 + 1/m)}{mb - 4\kappa b},\qquad m > \max(4\kappa, b\kappa),
\]
where $\kappa = L / \mu$ is the condition number, then Algorithm~\ref{alg:SSBB} converges linearly in expectation with
\[
\mathbb{E}[f(x_k) - f(x^*)]\le \rho^k[f(x_0) - f(x^*)].
\]
\end{corollary}

\begin{proof}
Because SSBB is a special case of Algorithm~\ref{alg:minibatchSVRG}, then we can easily get
\[
\mathbb{E}[f(x_{k+1}) - f(x^*)] \le \left[\frac{b}{m\mu\eta^{\SSBB}_k(b - 4L\eta^{\SSBB}_k)} + \frac{4(m+1)L\eta^{\SSBB}_k}{m(b-4L\eta^{\SSBB}_k)}\right] \mathbb{E}[f(x_k) - f(x^*)]
\]
when $\eta^{\SSBB}_k < b/(4L)$ and $\eta^{\SSBB}_k < 1/L$. From Lemma~\ref{etabound:SSBB}, this is valid for $m > \max(4\kappa, b\kappa)$. Also from Lemma~\ref{etabound:SSBB}, we have
\begin{align*}
    \rho^{\SSBB}_k &= \frac{b}{m\mu\eta^{\SSBB}_k(b - 4L\eta^{\SSBB}_k)} + \frac{4(m+1)L\eta^{\SSBB}_k}{m(b-4L\eta^{\SSBB}_k)} \\
    &\le \frac{b}{m\mu\frac{b}{mL}(b - 4L\frac{b}{m\mu})} + \frac{4(m+1)L\frac{b}{m\mu}}{m(b - 4L\frac{b}{m\mu})}\\
    & = \frac{\kappa m + 4\kappa b(1 + 1/m)}{mb - 4\kappa b}
\end{align*}
Hence if $m$ and $b$ are chosen so that $\rho < 1$, we have
\[
\mathbb{E}[f(x_k) - f(x^*)]\le \rho^k[f(x_0) - f(x^*)]
\]
as required.
\end{proof}
\textbf{Remark:} From Corollary~\ref{linconv:SSM} and \ref{linconv:SSBB}, the total computational complexity of gradients are both $O\bigl((n + \kappa^2)\log(1/\epsilon)\bigr)$. Actually, in Corollary~\ref{linconv:SSBB}, if the coefficient of the SSBB learning rate is replaced by $b/m^\alpha$, the proof will give $O\bigl((n + \kappa^2)\log(1/\epsilon)\bigr)$ complexity for any choice of $\alpha\in[0,1]$, which matches the result in Corollary~\ref{linconv:SSM}. Additionally, linear convergence rate can not be preserved if $b/m$ is replaced by $1/m$, since $\rho < 1$ can not be guaranteed.

\section{Proximal variant}\label{sec:nondiff}

As shown in \cite{xiao2014proximal}, SGD and SVRG may be easily extended to cover nondifferentiable objective functions of the form
\begin{equation}\label{prox large sum}
F(x) = f(x) + R(x) = \frac{1}{n}\sum_{i=1}^nf_i(x) + R(x),
\end{equation}
where $f$ satisfies Assumption~\ref{ass} and $R$ is a nondifferentiable function such as $R(x) = \|x\|_1$. In this section we will see that SSBB may likewise be extended, and the linear convergence is preserved.

To solve \eqref{prox large sum}, the proximal gradient method does
\[
x_k = \prox_{\eta R}\bigl(x_{k-1} - \eta \grad f(x)\bigr),
\]
with a \emph{proximal map}  defined by
\[
\prox_{R}(y) = \argmin_{x\in\mathbb{R}^d}\biggl\{\frac{1}{2}\|x - y\|^2 + R(x)\biggr\}.
\]
As in \cite{xiao2014proximal}, we replace the update rule  $x_{k,t+1} = x_{k,t} - \eta^{\SSBB}_kv_{k,t}$ in Algorithm~\ref{alg:SSBB} by
\begin{equation}\label{eq:prox}
x_{k,t+1} = \prox\bigl(x_{k,t} - \eta^{\SSBB}_kv_{k,t}\bigr).
\end{equation}
We will see that the resulting algorithm, which we will call prox-SSBB, remains linearly convergent as long as the following assumption holds.
\begin{assumption}\label{ass2}
The function $R$ is $\mu$-strongly convex in the sense that
\[
R(y) \ge R(x) + g(x)^\tp(y-x) + \frac{\mu}{2}\|y - x\|^2
\]
for all $x\in \dom(R)$, $g(x)\in\partial R(x)$, $y\in\mathbb{R}^d$, and $R(y) \coloneqq +\infty$ whenever $y \notin \dom(R)$. Here $\partial R(x)$ denotes subgradient at $x$.
\end{assumption}
It is a standard fact \cite[p.~340]{Rock} that if $R$ is a closed convex function on $\mathbb{R}^d$, then
\begin{equation}\label{Nonexpansiveness}
\|\prox_R(x) - \prox_R(y)\| \le \|x - y\|
\end{equation}
for all $x,y\in\text{dom}(R)$. We will write $\mu_f$ for the convexity parameter of $f$ in Assumption~\ref{ass} and  $\mu_R$ for that of $R$ in Assumption~\ref{ass2}. This implies that the overall objective function $F$ is strongly convex with $\mu \ge \mu_f + \mu_R$.

To establish linear convergence for prox-SSBB, we need an analogue of Lemma~\ref{restbound}, which is provided by \cite[Lemma~3]{xiao2014proximal}, reproduced here for easy reference.
\begin{lemma}[Xiao--Zhang]\label{prox restbound}
Let $f$ be as in Assumption~\ref{ass}, $R$ as in Assumptions~\ref{ass}, and $F = f + R$ with $x^*\coloneqq \argmin_{x} F(x)$.  Let $\Delta_{k,t} \coloneqq v_{k,t} - \grad f(x_{k,t})$ and
\[
g_{k,t} \coloneqq \frac{1}{\eta_k}(x_{k,t} - x_{k,t+1}) = \frac{1}{\eta_k}\bigl(x_{k,t} - \prox_{\eta_kR}(x_{k,t} - \eta_k v_{k,t})\bigr).
\]
If $0 < \eta_k < 1/L$, then
\begin{multline*}
F(x^*) \ge  \ F(x_{k,t+1}) + g_{k,t}^\tp(x^* - x_{k,t}) + \frac{\eta_k}{2}\|g_{k,t}\|^2\\
 + \frac{\mu_f}{2}\|x_{k,t} - x^*\|^2 + \frac{\mu_R}{2}\|x_{k,t+1} - x^*\|^2 + \Delta_{k,t}^\tp(x_{k,t+1} - x^*).
\end{multline*}
\end{lemma}
\begin{corollary}[Linear convergence of prox-SSBB]
Let $F$ and $x^*$ be as in Lemma~\ref{prox restbound} and $\eta_k =\eta^{\SSBB}_k$.  Then Corollary~\ref{linconv:SSBB} holds verbatim with $F$ in place of $f$.
\end{corollary}
\begin{proof}
To apply Lemma~\ref{prox restbound}, we need $\eta_k\le 1/L$ and this holds as we have $c(b) \le \mu/L$ among the assumptions of Lemma~\ref{etabound:SSBB}. In the notations of Lemma~\ref{prox restbound}, the update \eqref{eq:prox} is equivalent to $x_{k,t+1} = x_{k,t} - \eta_kg_{k,t}$. So
\[
\|x_{k,t+1} - x^*\|^2 = \|x_{k,t} - x^*\|^2 - 2\eta_kg_{k,t}^\tp (x_{k,t} - x^*) + \eta_k^2\|g_{k,t}\|^2.
\]
By Lemma~\ref{prox restbound}, we have
\begin{multline*}
- g_{k,t}^\tp (x_{k,t} - x^*) + \frac{\eta_k}{2}\|g_{k,t}\|^2\\
\le \ F(x^*) - F(x_{k,t+1}) - \frac{\mu_f}{2}\|x_{k,t} - x^*\|^2 - \frac{\mu_R}{2}\|x_{k,t+1}-x^*\|^2 - \Delta_{k,t}^\tp(x_{k,t+1} - x^*).
\end{multline*}
Therefore,
\[
\|x_{k,t+1} - x^*\|^2 \le \|x_{k,t} - x^*\|^2 - 2\eta_k\Delta_{k,t}^\tp(x_{k,t+1} - x^*) + 2\eta_k[F(x^*) - F(x_{k,t+1})].
\]
We bound the middle term on the right. Let $\bar{x}_{k,t+1} \coloneqq  \prox_{\eta_kR}(x_{k,t} - \eta_k\nabla f(x_{k,t}))$. Then
\[
\begin{aligned}
- 2\eta_k\Delta_{k,t}^\tp(x_{k,t+1} - x^*) & = -2\eta_k\Delta_{k,t}^\tp(x_{k,t+1} - \bar{x}_{k,t+1}) - 2\eta_k\Delta_{k,t}^\tp(\bar{x}_{k,t+1} - x^*)\\
& \le 2\eta_k\|\Delta_{k,t}\|\|x_{k,t+1} - \bar{x}_{k,t+1}\| - 2\eta_k\Delta_{k,t}^\tp(\bar{x}_{k,t+1} - x^*)\\
& \le 2\eta_k\|(x_{k,t} - \eta_kv_{k,t}) - (x_{k,t} - \eta_k\nabla f(x_{k,t}))\| - 2\eta_k\Delta_{k,t}^\tp(\bar{x}_{k,t+1} - x^*)\\
& = 2\eta_k^2\|\Delta_{k,t}\|^2 - 2\eta_k\Delta_{k,t}^\tp(\bar{x}_{k,t+1} - x^*),
\end{aligned}    
\]
where the first inequality is Cauchy--Schwarz and the second follows from Lemma~\ref{Nonexpansiveness}. The remaining steps are as in the proofs of Theorem~\ref{maintheorem} and Corollary~\ref{linconv:SSBB} with $F$ in place of $f$.
\end{proof}

\section{Numerical Experiments}\label{sec:num}

As mentioned earlier, our method of choice is Algorithm~\ref{alg:SSBB}, the stochastic Steffensen--Barzilai--Borwein method (SSBB)  with minibatching. We will compare it with several benchmarking algorithms: stochastic gradient descent (SGD), stochastic variance reduced gradient (SVRG) \cite{johnson2013accelerating}, stochastic LBFGS \cite{moritz2016linearly}, and the first two with Barzilai--Borwein step size (SGD--BB and SVRG--BB) \cite{tan2016barzilai}. We tests these algorithms on popular empirical risk minimization problems --- ridge regression, logistic regression and support vector machines with squared hinge loss --- on standard datasets in \texttt{LIBSVM}. The parameters involved are summarized in Table~\ref{information}. Our experiments show that SSBB compares favorably with these benchmark algorithms. All our codes are available at \url{https://github.com/Hs-DeeMo/stochastic-steffensen}.
\begin{table}[htb]
\centering
  \begin{tabular}{lllllll}
    \toprule
    Data set     & Loss function     & $n$ & $d$ & $m$  & $b$ & $\lambda$\\
    \midrule
    Synthetic & Squared loss  & $5000$  & $100$ & $2n$  & $4$   & $10^{-4}$\\
    \texttt{w6a}     & Logistic loss & $17188$ & $300$ & $n$ & $16$    & $10^{-4}$ \\
    \texttt{a6a}     & Squared hinge loss  & $11220$ & $123$ & $n$ & $16$ & $10^{-3}$\\
    \bottomrule
  \end{tabular}
\medskip
  \caption{List of experiments. Sample size $n$, dimension $d$, batch size $b$, regularization parameter $\lambda$.}
  \label{information}
\end{table}

For a fair comparison, all algorithms are minibatched. We set a batch size of $b = 4$ for ridge regression, $b = 16$ for logistic loss and squared hinge loss. The inner loop size is set at $m = 2n$ or $n$ according to Table~\ref{information}. The learning rates in SGD, SVRG, and SLBFGS are hyperparameters that require separate tuning; we pick the best possible values with a grid search. SLBFGS requires more hyperparameters: As suggested by the authors of \cite{moritz2016linearly}, we set the Hessian update interval to be $L = 10$, Hessian batch size to be $b_H= L b$, and memory length to be $M = 10$.  All experiments are initialized with $x_0 =0$. We repeat every experiment ten times and report average results.

In all figures, we present the convergence trajectory of each method. The vertical axis represents in log scale the value $f(x_k) - f(x^*)$ where we estimate $f(x^*)$ by running full gradient descent or Newton method multiple times. The horizontal axis represents computational cost as measured by either number of gradient computations divided by $n$ or the actual running time --- we present both. In all experiments, we note that the convergence trajectories of SSBB agree with the linear convergence established in Section~\ref{sec:conv}. 

\subsection{Ridge Regression} Figure~\ref{ridge} shows a simple ridge regression on a synthetic dataset generated in a controlled way to give us the true global solution. We generate $x^* \in \mathbb{R}^d$ with $x^*_i\sim\mathcal{N}(0,1)$ and $A\in\mathbb{R}^{n\times d}$ with $a_{ij}  \sim\mathcal{N}(0,1)$. We form $y = Ax^* + b$ with $b$ an $n$-dimensional standard normal variate. We then attempt to recover $x^*$ from $A$ and $y$ by optimizing, with $\lambda = 10^{-4}$,
\[
\min_{x\in\mathbb{R}^d}\frac{1}{n}\sum_{i=1}^n\|y - Ax\|^2 + \frac{\lambda}{2}\|x\|^2.
\]

\begin{figure}[htb]
\centering
\includegraphics[width=0.48\columnwidth,trim=0 0 30 30, clip]{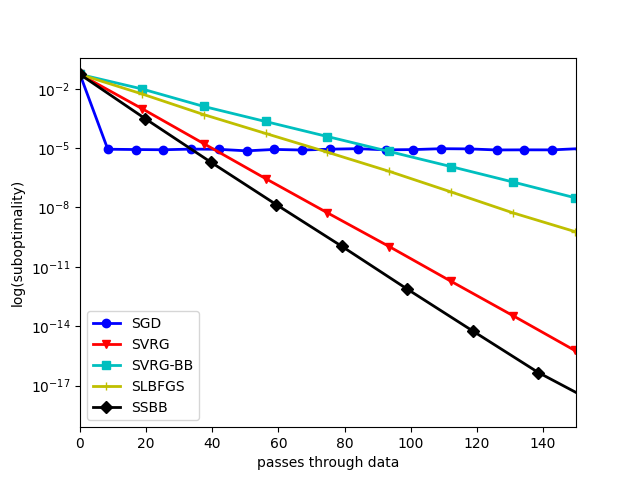}
\includegraphics[width=0.48\columnwidth,trim=0 0 30 30, clip]{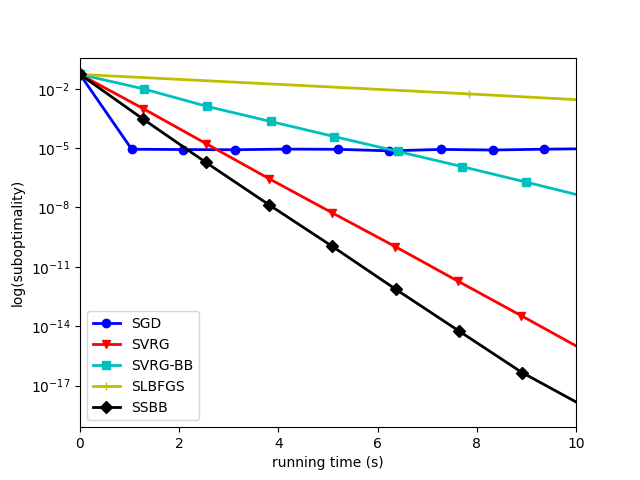}
\caption{Ridge regression on synthetic dataset with regularization parameter $\lambda = 10^{-4}$. \emph{Left:} number of passes through data. \emph{Right:} running time.}
\label{ridge}
\end{figure}

\subsection{Logistic Regression}
Figure~\ref{logistic} shows the results of a binary classification problem on the on \texttt{w6a} dataset from \texttt{LIBSVM} using an $l^2$-regularized binary logistic regression. The associated optimization problem with regularization $\lambda = 10^{-4}$ and labels $y_i\in \{-1, +1\}$ is
\[
\min_{x\in\mathbb{R}^d}\frac{1}{n}\sum_{i=1}^n\log\bigl(1 + e^{-y_i(a_i^\tp x)}\bigr) + \frac{\lambda}{2}\|x\|^2.
\]

\begin{figure}[htb]
\centering
\includegraphics[width=0.48\columnwidth,trim=0 0 30 30,clip]{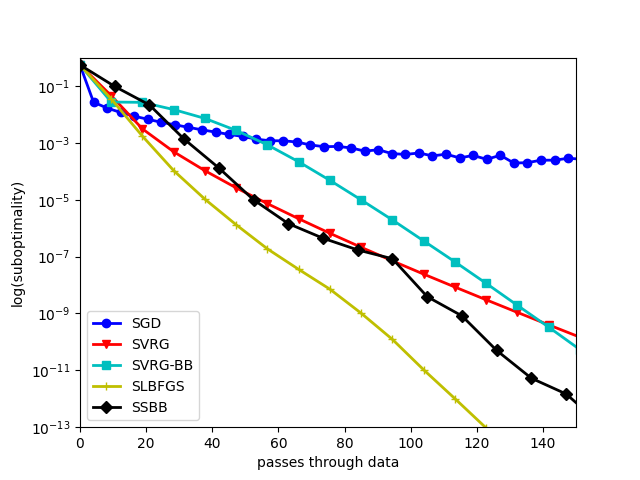}
\includegraphics[width=0.48\columnwidth,trim=0 0 30 30,clip]{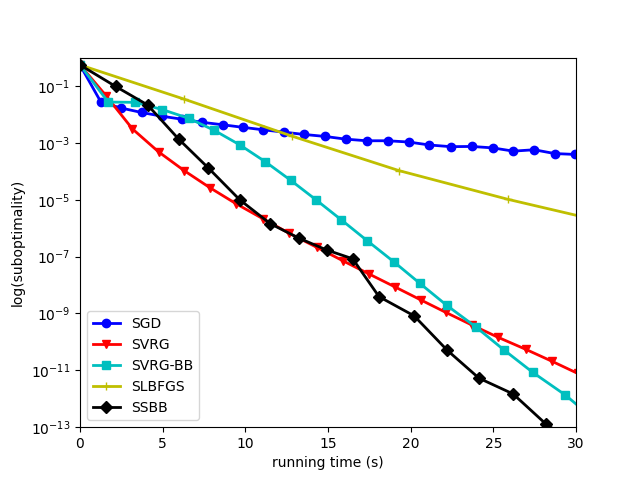}
\caption{$l^2$-regularized logistic regression on \texttt{w6a} dataset from \texttt{LIBSVM} with regularization parameter $\lambda = 10^{-4}$. \emph{Left:} number of passes through data. \emph{Right:} running time.}
\label{logistic}
\end{figure}

\subsection{Squared Hinge Loss}
Figure~\ref{hinge} shows the results of a support vector machine classifier with $l^2$-regularized squared hinge loss and $\lambda=10^{-3}$ on the \texttt{a6a} dataset from \texttt{LIBSVM}. The optimization problem in this case is
\[
\min _{x\in\mathbb{R}^d} \frac{1}{n} \sum_{i=1}^{n}[(1-y_{i} a_{i}^\tp  x)_{+}]^2+\frac{\lambda}{2}\|x\|^2.
\]

\begin{figure}[htb]
\centering
\includegraphics[width=0.48\columnwidth,trim=0 0 30 30,clip]{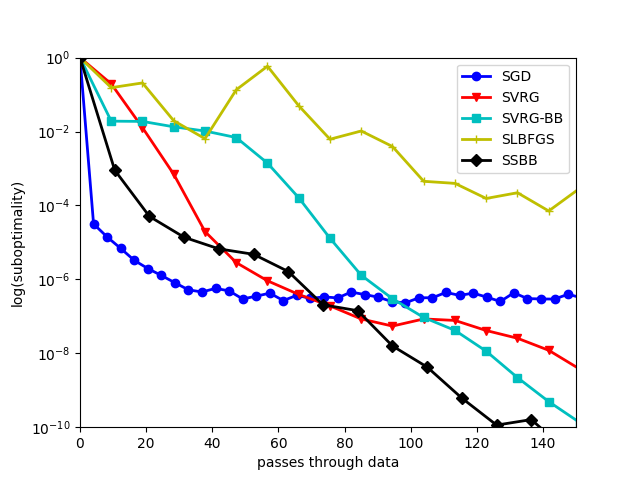}
\includegraphics[width=0.48\columnwidth,trim=0 0 30 30,clip]{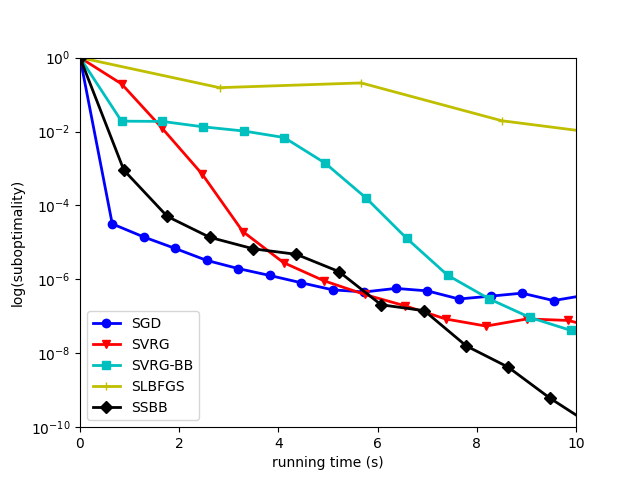}
\caption{$l^2$-regularized squared hinge loss on \texttt{a6a} from \texttt{LIBSVM} with regularization parameter $\lambda = 10^{-3}$. \emph{Left:} number of passes through data. \emph{Right:} running time.}
\label{hinge}
\end{figure}

The results are clear: SSBB solves the problems to high levels of accuracy and is the fastest, whether measured by running time or by number of passes through data, in all but one case. The only exception  is shown on the left of Figure~\ref{logistic}, where SLBFGS is better when measured by the number of passes through data. But even in this case, SSBB is still the second best. Moreover, when  measured in actual running time as shown on the right of Figure~\ref{logistic}, SSBB becomes the fastest whereas SLBFGS drops to the fourth place. This is consistent with our discussion in Section~\ref{sec:intro}, namely, SLBFGS incurs additional computational cost due to its matrix-vector products that SSBB completely avoids. For the other two experiments in Figures~\ref{ridge} and \ref{hinge}, SSBB beats all methods in both measures of speed.

\section{Conclusion}

The stochastic Steffensen methods introduced in this article are (i) simple to implement, (ii) efficient to compute, (iii) easy to incorporate, (iv) tailored for massive data and high dimensions, have (v) minimal memory requirements and (vi) a negligible number of hyperparameters to tune. The last point is in contrast to more sophisticated methods involving moments \cite{duchi2011adaptive,hinton2012neural,kingma2014adam} or momentum \cite{Nesterov,Poljak,qian1999momentum,reddi2019convergence}, which require heavy tuning of many more hyperparameters. SSM and SSBB require just two --- minibatch size $b$ and inner loop size $m$. In fact, since we typically set $ m = \lfloor n / b\rfloor$, there is really just one hyperparameter $b$ to be tuned.

The point (iii) also deserves special mention. Since SSM and SSBB are ultimately encapsulated in the respective learning rates $\eta_k^{\SS}$ and $\eta_k^{\SSBB}$, they are versatile enough to be  incorporated into other methods such as those in \cite{duchi2011adaptive,hinton2012neural,kingma2014adam,Nesterov,Poljak,qian1999momentum,reddi2019convergence}, assuming that we are willing to pay the price in hyperparameters tuning. We hope to explore this in future work.

\subsection*{Acknowledgment}

This work is partially supported by DARPA HR00112190040, NSF DMS-1854831, and the Eckhardt Faculty Fund. LHL would like to thank Junjie Yue for helpful discussions.

\bibliographystyle{abbrv}
\bibliography{Steffensen}

\end{document}